\documentclass[leqno]{amsart}
\usepackage{amssymb,latexsym}
\usepackage{url}
\usepackage{newlattice}
\usepackage{microtype}

\newtheorem{theorem}{Theorem}
\newtheorem{lemma}[theorem]{Lemma}
\newtheorem{corollary}[theorem]{Corollary}
\newtheorem{proposition}[theorem]{Proposition}

\newtheorem{question}{Question}

\newcommand{\varV}{\mathbf{V}}
\newcommand{\varD}{\mathbf{D}}
\newcommand{\varL}{\mathbf{L}}
\newcommand{\cp}{^\mathrm{c}}
\newcommand{\SC}[1]{\msf{C}_{#1}}
\DeclareMathOperator{\Free}{Free}

\begin{document}

\begin{center}
\texttt{Comments, corrections, and related references welcomed,
as always!}\\[.5em]
{\TeX}ed \today
\vspace{2em}
\end{center}

\title{Isotone maps on lattices}
\thanks{This preprint is accessible online at
\url{http://math.berkeley.edu/~gbergman/papers/}
}


\author{G. M. Bergman}
\email[G. M. Bergman]{gbergman@math.berkeley.edu}
\urladdr[G. M. Bergman]{http://math.berkeley.edu/~gbergman/}
\address{University of California\\
         Berkeley, CA\\
         USA}

\author{G. Gr\"{a}tzer}
\email[G. Gr\"atzer]{gratzer@me.com}
\urladdr[G. Gr\"atzer]{http://server.math.umanitoba.ca/homepages/gratzer/}
\address{Department of Mathematics\\
         University of Manitoba\\
         Winnipeg, MB R3T 2N2\\
         Canada}

\begin{abstract}
Let $\E L = (L_i \mid i\in I)$ be a family of lattices in
a nontrivial lattice variety~$\varV$, and let
$\gf_i \colon L_i \to M$, for $i\in I$, be isotone maps
(not assumed to be lattice homomorphisms)
to a common lattice $M$ (not assumed to lie in $\varV$).
We show that the maps $\gf_i$ can be extended to
an isotone map $\gf\colon L \to M$,
where $L=\Free_{\varV}\E L$ is the free product of the $L_i$ in $\varV$.
This was known for
$\varL=\varV$, the variety of all lattices (Yu.\,I. Sorkin 1952).

The above free product $L$ can be viewed as the
free lattice in $\varV$ on the partial lattice $P$ formed
by the disjoint union of the $L_i$.
The analog of the above result does not, however, hold for
the free lattice $L$ on an arbitrary partial lattice $P$.
We show that the only codomain lattices $M$ for which
that more general statement holds are the complete lattices.
On the other hand, we prove the analog of our main result
for a class
of partial lattices $P$ that are not-quite-disjoint
unions of lattices.

We also obtain some results similar to our main one, but
with the relationship \tbf{lattices\,:\,orders} replaced
either by \tbf{semilattices\,:\,orders}
or by \tbf{lattices\,:\,semilattices}.

Some open questions are noted.
\end{abstract}
\maketitle

\section{Introduction}\label{S.intro}
By Yu.\,I. Sorkin \cite[Theorem~3]{sorkin},
if $\E L = (L_i \mid i\in I)$ is a family of lattices
and $\gf_i\colon L_i\to M$ are isotone maps of the lattices $L_i$
into a lattice $M$, then there exists an
isotone map $\gf$ from the
free product $\Free\E L$ of the $L_i$ to $M$
that extends all the $\gf_i$.
(For a clarification of Sorkin's proof, see A.\ Kravchenko \cite{AK};
and for an alternative, simpler proof,
G. Gr\"{a}tzer, H. Lakser and C.\,R. Platt \cite[\S4]{GLP70}.)

Our main result, proved in
Section~\ref{S.main}, is a generalization of this fact,
with $\Free\E L$ replaced by $\Free_{\varV}\E L$, the
free product of the $L_i$ in any nontrivial variety $\varV$ of lattices
containing them --- though not necessarily containing~$M$.
(In Section~\ref{S.alt_pfs}, we explore some variants of our proof of
this result.)

We may regard $\Free_{\varV}\E L$ as the free
lattice in $\varV$ on the partial lattice $P$ given by the disjoint
union of the $L_i$.
Does the analog of the above result hold for more general
partial lattices $P$ and their free lattices $L$?
In Section~\ref{S.complete} we find that the lattices $M$ such that
this statement holds for \emph{all} partial lattices $P$ are
the complete lattices.
On the other hand, we describe in Section~\ref{S.retracts}
a class of partial lattices $P$, related to but distinct
from the class considered in Section~\ref{S.main},
for which the full analog of the result of that section holds.

Since \emph{semilattices} lie between orders and lattices,
it is plausible that statements similar to our main result should hold,
either with ``lattice'' weakened to ``semilattice'', or with
``lattice'' unchanged but ``isotone map''
strengthened to ``semilattice homomorphism''.
In~Section~\ref{S.semilat1} we shall find that the
former statement is easy to prove.
In that section and Section~\ref{S.semilat2},
we obtain several approximations to the latter statement;
we do not know whether the full statement holds.

The reader familiar with the concepts of \emph{quasi}variety and
\emph{pre}variety will find that the proofs given
in this note for varieties of lattices in fact work
for those more general classes.
However, varieties are not sufficient for the result of
Section~\ref{S.semilat2}, so we develop that in terms of prevarieties.

In Section~\ref{S.questions} we note some open questions.

For general definitions and results in lattice theory
see \cite{GLT2} or \cite{GLT3}.

\section{Extending isotone maps to free product lattices}\label{S.main}

Let $\varV$ be a nontrivial lattice variety, that is, a variety $\varV$
of lattices having a member with more than one element.
Let $\E L = (L_i \mid i\in I)$ be a family of lattices in~$\varV$,
and $L = \Free_{\varV}\E L$  their free product in $\varV$.
Finally, let $(\gf_i\colon L_i\to M\mid i\in I)$ be a family
of isotone maps into a lattice $M$, not assumed to lie in $\varV$.

To show that the $\gf_i$ have a common extension to $L$,
it suffices, by the universal property of
$L$, to find some $L'\in\varV$ such that
each map $\gf_i$ factors $L_i\to L'\to M$, where the first map is
a lattice homomorphism, and the second an
isotone map not depending on $i$.
So~let us, for now, forget free products,
and obtain such a lattice~$L'$.

We first note (as remarked in~\cite[last paragraph]{GLP70}
for $\varV=\varL$) that this is easy if $M$ has a least
element, or more generally, if its subsets $\gf_i(L_i)$ have
a common lower bound $e\in M$.
In that case, we begin by enlarging all the $L_i$ to lattices
$\bar{L}_i=\{e_i\}+L_i$, where $e_i$ is a new least
element, and extend the $\gf_i$
to maps $\bar{\gf}_i\colon \bar{L}_i\to M$ mapping $e_i$ to $e$.
Now let $L'$ be the sublattice of $\prod(\bar{L}_i \mid i\in I)$
consisting of the elements $x=(x_i\mid i\in I)$ such that
$x_i=e_i$ for all
but finitely many $i$;
and let us map each $L_i$ into $L'$ by the homomorphism
sending $x\in L_i$ to the element having $i$-component $x$, and
$j$-component $e_j$ for all $j\neq i$.
We now map $L'$ to $M$ using the isotone map $\gy$ given by
\begin{equation}\begin{minipage}[c]{25pc}\label{d.pre_gy}
$\gy(x)=\JJm{\bar{\gf}_i(x_i)}{i\in I}$ \quad
for $x=(x_i\mid i\in I)\in L'$.
\end{minipage}\end{equation}
This infinite join is defined because all but finitely many of
the joinands are $e$; and it is easy to verify that for each
$i$, the composite
map $L_i\to L'\to M$ is the given isotone map $\gf_i$, as required.

If, rather, the $\gf_i(L_i)$ have a common \emph{upper} bound,
the dual construction is, of course, available.

In the absence of either sort of bound, we shall follow
the same pattern of adjoining to the $L_i$ elements
$e_i$ with a common image $e$ in $M$ (this time an arbitrary
element of that lattice); but that construction takes a bit
more work, as does the one analogous to the definition~(\ref{d.pre_gy})
of the isotone map $\gy:L'\to M$.
The first of these steps is carried out in the following lemma
(where $L$ corresponds to the above~$L_i$).

\begin{lemma}\label{L.L&e}
Let $\gf\colon L\to M$ be any isotone map of lattices, and $e$ any
element of $M$.
Then there exists a lattice extension $\bar{L}$ of $L$,
and an isotone map
$\bar{\gf}\colon \bar{L}\to M$ extending $\gf$,
such that $e\in\bar{\gf}(\bar{L})$.
Moreover, $\bar{L}$ can be taken to lie in any nontrivial lattice
variety $\varV$ containing $L$.
\end{lemma}

\begin{proof}
Let $\bar{L}=L\times \SC 2 \times \SC 2$,
where $\SC 2$ is the $2$-element lattice
$\{0,1\}$, and embed $L$ in $\bar{L}$ by $x\mapsto(x,0,1)$.
Define $\bar{\gf}\colon \bar{L}\to M$ by\vspace{.3em}
\begin{equation}\begin{minipage}[c]{25pc}\label{d.gf}
$\begin{cases}
\begin{array}{rrc}
\bar{\gf}(x,0,1) &=& \gf(x),\\[.3em]
\bar{\gf}(x,1,0) &=& e,\\[.3em]
\bar{\gf}(x,0,0) &=& \gf(x)\wedge e,\\[.3em]
\bar{\gf}(x,1,1) &=& \gf(x)\vee e.
\end{array}
\end{cases}$\vspace{.3em}
\end{minipage}\end{equation}
It is easy to check that $\bar{\gf}$ is isotone, and it
clearly has $e$ in its range.
Since~$\SC 2=\{0,1\}$ belongs to every nontrivial variety of lattices,
$\bar{L}$ will belong to any nontrivial variety~$\varV$ containing~$L$.
\end{proof}

The next lemma gives the construction we
will use to weld our $I$-tuple of isotone maps into one map.

\begin{lemma}\label{L.M'toM}
Let $M$ be a lattice, $e$ any element of $M$, and $I$
a nonempty set.
Let $M'$ be the sublattice of $M^I$ consisting
of those elements $f$ such that $f(i)=e$ for all but finitely many
$i\in I$.
Then there exists a map $\gy\colon M'\to M$ such that
\begin{equation}\begin{minipage}[c]{25pc}\label{d.isotone}
$\gy$ is isotone
\end{minipage}\end{equation}
and
\begin{equation}\begin{minipage}[c]{25pc}\label{d.all_but_one}
For every $i\in I$, and every $f\in M'$ satisfying
$f(j)=e$ for all $j\neq i$, we have $\gy(f) = f(i)$.
\end{minipage}\end{equation}
\end{lemma}

\begin{proof}
For $f\in M'$, define
\begin{equation}\begin{minipage}[c]{25pc}\label{d.gy}
$\gy(f) =
  \begin{cases}
  \MMm{f(i)}{i\in I} & \text{if $f(i)\leq e$ for all $i\in I$,}\\[.3em]
  \JJm{f(i)}{i\in I,\ f(i) \nle e} & \text{otherwise.}
  \end{cases}$
\end{minipage}\end{equation}
These meets and joins are defined because for each $f \in M'$,
there are only finitely many distinct values $f(i)$.

It is easy to see that $\gy$ satisfies~(\ref{d.all_but_one}).
To obtain~(\ref{d.isotone}), observe that
\begin{equation}\begin{minipage}[c]{25pc}\label{d.subset}
For $f\leq g$ in $M'$, we have
$\{i\mid f(i)\not\leq e\}\ \ci\ \{i\mid g(i)\not\leq e\}$.
\end{minipage}\end{equation}
Hence given $f\leq g$, there are three possibilities: Either the
definitions of $\gy(f)$ and $\gy(g)$ both fall under the first case
of~(\ref{d.gy}), or they both fall under the second,
or that of $\gy(f)$ falls under the first and
that of $\gy(g)$ under the second.

If both fall under the first case, then $\gy(f)\leq\gy(g)$
because the meet operation of $M$ is isotone.

If both fall under the second, the same conclusion follows
using the fact that the join operation is isotone, together
with the fact that bringing in more joinands, as can happen in
view of~(\ref{d.subset}), yields a join greater than or
equal to what we would get without those additional terms.

Finally, if the evaluation of $\gy(f)$ falls under the first
case and that of $\gy(g)$ under the second, we may choose,
in view of the latter fact, an $i$ such that $g(i)\not\leq e$.
Then
\begin{equation}\begin{minipage}[c]{25pc}\label{d.leqleq}
$\gy(f)\ \leq\ f(i)\ \leq\ g(i)\ \leq\ \gy(g)$,
\end{minipage}\end{equation}
completing the proof of~(\ref{d.isotone}).
\end{proof}

We can now fill in the proof of our main theorem.

\begin{theorem}\label{T.main}
Let $\varV$ be a nontrivial variety of lattices,
$\E L = (L_i \mid i\in I)$ a family of lattices in $\varV$,
and $(\gf_i\colon L_i\to M\mid i\in I)$
a family of isotone maps from the $L_i$ to a
lattice $M$ not necessarily in $\varV$.
Then there exists an isotone map $\gf\colon \Free_{\varV}\E L\to M$
whose restriction to each $L_i\ci \Free_{\varV}\E L$ is $\gf_i$.
\end{theorem}

\begin{proof}
Choose any element $e\in M$, and extend each $\gf_i$
as in Lemma~\ref{L.L&e} to a map $\bar{\gf}_i\colon \bar{L}_i\to M$
on a lattice extension $\bar{L}_i\ce L_i$ in $\varV$,
so that some $e_i\in\bar{L}_i$ is mapped by
$\bar{\gf}_i$ to $e\in M$.
Now map each $L_i$ into $\prod(\bar{L}_i\mid i \in I)$
by sending every element $x\in L_i$ to the element having $i$-th
coordinate $x$, and $j$-th coordinate $e_j$ for all $j\neq i$.
These maps are lattice homomorphisms, hence together
they induce a
homomorphism $\Free_{\varV}\E L\to\prod(\bar{L}_i \mid i\in I)$.
Moreover, this map
has range in the sublattice $L'$ of elements whose $j$-coordinates
are $e_j$ for almost all $j$, since the image of each $L_i$
lies in that sublattice.

Mapping $\prod\bar{L}_i$ to $M^I$ by the isotone map
$\prod\bar{\gf}_i$, we see that the
above sublattice $L'\ci\prod\bar{L}_i$ is carried
into the sublattice $M'\ci M^I$ of Lemma~\ref{L.M'toM}.
Bringing in the isotone map $f\colon M'\to M$ of
that lemma, we get our desired isotone
map $\gf$ as the composite $\Free_{\varV}\E L\to L'\to M'\to M$.
It follows from~(\ref{d.all_but_one}) that the restriction of $\gf$
to each $L_i$ is $\gf_i$.
\end{proof}

We note a curious consequence of the fact that the $M$ of
Theorem~\ref{T.main} need not lie in $\varV$.

\begin{corollary}\label{C.VtoL}
Let $\E L = (L_i \mid i\in I)$ be a family of lattices in
a nontrivial lattice variety~$\varV$,
let $\Free_{\varV}\E L$ be their free product in $\varV$,
and let $\Free\E L$ be their free product in
the variety $\varL$ of all lattices.
Then there exists an isotone map $\Free_{\varV}\E L\to\Free\E L$
which acts as the identity on each $L_i$.

In particular, for any nontrivial lattice variety $\varV$ and
any set $X$, there exists an isotone map
$\Free_{\varV}(X)\to\Free(X)$ \textup{(}where
these denote the free lattices on the set $X$ in $\varV$
and in $\varL$ respectively\textup{)},
which acts as the identity map on $X$.
\end{corollary}

\begin{proof}
For the first statement, apply Theorem~\ref{T.main} to the
inclusions of the $L_i$ in~$\Free\E L$.
The second is the case of the first where all $L_i$
are one-element lattices.
\end{proof}

\section{Digression: sketches of some alternate proofs of Theorem~\ref{T.main}}\label{S.alt_pfs}

The definition~(\ref{d.gy})
of the isotone map $\gy$ used in the proof of Lemma~\ref{L.M'toM}
is clearly asymmetric in the meet and join operations.

We sketch below a variant proof of Theorem~\ref{T.main} which uses a
function that is symmetric in these operations --- but lacks instead
(when $|I|>2$) the symmetry in the family of lattices $L_i$
which the proof given above clearly has.
We shall then show that one cannot have it both ways:
a map of the required sort
having both sorts of symmetry does not, in general, exist.
However, we show that we \emph{can} get such a map
if $M$ lies in the given variety~$\varV$.

This section will be sketchier than the rest of the paper.
In particular, we will be informal about our two sorts of symmetry;
though in the next-to-last paragraph we will
indicate how to make these considerations precise.

Our new proof of Theorem~\ref{T.main} starts with
a generalization of the construction of Lemma~\ref{L.L&e}.
Namely, suppose we are given isotone maps of \emph{two} lattices into a
common lattice, $\gf_i\colon L_i\to M$ for $i=0,1$.
Let
\begin{equation}\begin{minipage}[c]{25pc}\label{d.L0L1C2C2}
$L'\ =\ L_0\times L_1\times\SC 2\times\SC 2$.
\end{minipage}\end{equation}
Then taking any $e_0\in L_0$, $e_1\in L_1$, we can embed our two
lattices in $L'$ by the homomorphisms
\begin{equation}\begin{minipage}[c]{25pc}\label{d.embed}
$L_0\to L'$\quad acting by\quad $x \mapsto(x,e_1,1,0)$,\\[.3em]
$L_1\to L'$\quad acting by\quad $y \mapsto(e_0,y,0,1)$.
\end{minipage}\end{equation}
Now define the isotone map $\gf'\colon L'\to M$ by
\begin{equation}\begin{minipage}[c]{25pc}\label{d.gf'}
$\begin{cases}
\begin{array}{rlc}
\gf'(x,y,1,0) &=& \gf_0(x),\\
\gf'(x,y,0,1) &=& \gf_1(y),\\
\gf'(x,y,0,0) &=& \gf_0(x) \mm \gf_1(y),\\
\gf'(x,y,1,1) &=& \gf_0(x) \jj \gf_1(y).
\end{array}
\end{cases}$
\end{minipage}\end{equation}
Clearly, $\gf'$ acts on the embedded images of
the $L_i$ by the $\gf_i$; and
as before, since $\SC 2$ belongs to every nontrivial variety of
lattices, $L'$ belongs to any nontrivial variety~$\varV$
containing the~$L_i$.

This, in fact, gives us Theorem~\ref{T.main} for $|I|=2$ by a
construction
symmetric both in meet and join, and in our family of lattices.

Now suppose more generally that we have lattices
$L_i\in\varV$ and isotone maps $\gf_i\colon L_i\to M$
indexed by an arbitrary set $I$.
Assuming without loss of generality that $I$ is an ordinal,
we shall construct the desired $L'\in\varV$ and
isotone map $\gf'\colon L'\to M$ by a recursive transfinite
iteration of the above construction.
It is the recursion that will lose us our symmetry in the $L_i$,
via the arbitrary choice of an identification of $I$ with an
ordinal, i.e., of a well-ordering on~$I$.

To describe the recursion, let $1<k\leq I$, and
assume that we have constructed lattices
$L'_{(j)}$ for all $1\leq j<k$, which satisfy
\begin{equation}\begin{minipage}[c]{25pc}\label{d.L2...Lj}
$L_0\ =\ L'_{(1)}\ \ci\ L'_{(2)}\ \ci\ \dots
\ \ci\ L'_{(j)}\ \ci\ \dots$,
\end{minipage}\end{equation}
together with lattice embeddings $L_i\to L'_{(j)}$ for $i<j$, and
isotone maps $L'_{(j)}\to M$,
and that these form a coherent system, in the sense that for $i<j<j'$,
the composite $L_i\to L'_{(j)}\ci L'_{(j')}$ is the
embedding $L_i\to L'_{(j')}$, and
the composite $L'_{(j)}\ci L'_{(j')}\to M$
is the isotone map $L'_{(j)}\to M$;
and, finally, such that for every $i<j$, the composite
$L_i\to L'_{(j)}\to M$ is the given isotone map $\gf_i$.

If $k$ is a successor ordinal, $k=j+1$,
we apply the $|I|=2$ case of our construction, described
in~(\ref{d.L0L1C2C2})-(\ref{d.gf'}), to the pair of lattices
$L'_{(j)}$ and $L_{j}$ and their isotone maps to~$M$,
calling the resulting lattice
\begin{equation}\begin{minipage}[c]{25pc}\label{d.L'k=}
$L'_{(j+1)}\ =\ L'_{(j)}\times L_{j}\times \SC 2\times\SC 2$,
\end{minipage}\end{equation}
and identifying $L'_{(j)}$ with its image therein under
the first map of~(\ref{d.embed}).
If, on the other hand, $k$ is a limit ordinal, we let
$L'_{(k)}$ be the union of the $L'_{(j)}$ over all $j<k$.
In each case, the asserted properties are immediate.
Thus, we can carry our construction up to $k=I$,
the resulting lattice $L'_{(I)}$ being our desired $L'$.

What are the consequences of the different kinds of symmetry of the
construction of the preceding section and the one just sketched?

Because the former was symmetric in the $L_i$, we can deduce,
for instance, that
in the final statement of Corollary~\ref{C.VtoL}, if $X$ is finite,
then the isotone map $\Free_{\varV}(X)\to\Free(X)$ can be taken to
respect the actions of the symmetric group $\operatorname{Sym}(X)$
on these two lattices.
(Why assume $X$ to be finite?
So that in applying Lemma~\ref{L.L&e}, we can choose an
$e\in M=\Free(X)$ invariant under that group action, say the join
of the given generators.
Alternatively, without this finiteness assumption,
if we choose any $x_0\in X$ and
perform our construction with $e=x_0$, we can get our map to
respect the action of $\operatorname{Sym}(X-\{x_0\})$.)

On the other hand, using our new construction we can deduce
that if $\varV$ is closed under taking dual lattices,
we can, instead, in that same final statement of Corollary~\ref{C.VtoL},
take the isotone map $\Free_{\varV}(X)\to\Free(X)$ to
respect the anti-automorphisms of the domain and codomain that
fix the free generators but interchange meet and join.
(Again, we have to decide what to use for our distinguished elements
$e_0$, $e_1$ at each application of~(\ref{d.embed}).
In this case, we may, at each such step, take $e_0$ to be any
of the preceding generators, while
for $e_1$ we have no choice but to use the generator we are adjoining.)

Let us now show that
for $|I|=3$, we cannot get a construction with both
sorts of symmetry.
If we could, then letting $\varD$ denote the
variety of distributive lattices, we could get an isotone map
$\gf\colon\Free_{\varD}(3)\to\Free(3)$ respecting all permutations
of the generators, and also
the anti-automorphisms that interchange meets and joins.

Now $\Free_{\varD}(3)$ has an element invariant under all these
symmetries; namely, writing its three generators $a$, $b$, $c$,
the element
\begin{equation}\begin{minipage}[c]{25pc}\label{d.sym}
$(a\mm b)\jj(b\mm c)\jj(c\mm a)\ =\ (a\jj b)\mm(b\jj c)\mm(c\jj a)$.
\end{minipage}\end{equation}

On the other hand, for any set $X$,
the only elements of $\Free(X)$ that can be invariant
under an anti-automorphism are the given free generators;
this follows from the fact that every element
of $\Free(X)$ other than those generators is either
meet-reducible or join-reducible, but never both.
(Cf.\ \cite[condition (W) on p.477, and Corollary~534(iii)]{GLT3}.)
So $\Free(3)$ has no element with both sorts of symmetry
to which one could send the element given by~(\ref{d.sym}).

Finally, let us show that we \emph{can} get both sorts of
symmetry if the lattice $M$ lies in $\varV$.
(Of course, this restriction makes it impossible to use the result to
prove a version of Corollary~\ref{C.VtoL}.)
We record in the next lemma the raw construction used in the proof.
Though that lemma requires $M$ to lie in $\varV$, it
does not require the same of the $L_i$.
But it is easy to see that if we add the assumption that
the $L_i$ lie in $\varV$, the lattice $L'$ obtained will
lie there as well,
hence the construction will induce, as in Theorem~\ref{T.main},
an isotone map $\gf\colon\Free_{\varV}\E L\to M$
acting as $\gf_i$ on each $L_i$.
Moreover, the construction clearly has all the asserted symmetries.
(We remark that the factor $\Free_{\varV}(I)$ in the construction
reduces, when $|I|=2$, to the lattice $\SC 2\times\SC 2$
of~(\ref{d.L0L1C2C2}).
So one could say it was the fact that all nontrivial lattice
varieties have the same $2$-generator free lattice that allowed
us to get the doubly symmetric construction in that
two-lattice case with no added restriction on $M$.)

\begin{lemma}\label{L.prodtimesfree}
Let $\E L=(L_i \mid i\in I)$ be a family of lattices,
let $M$ be a lattice, and for each $i\in I$,
let $\gf_i\colon L_i\to M$ be an isotone map.

Let $\varV$ be a lattice variety containing~$M$, and
in the free lattice $\Free_{\varV}(I)$, let the $i$-th generator
be denoted $g_i$ for each $i$.
Let
\begin{equation}\begin{minipage}[c]{25pc}\label{d.L'=prod}
$L'\ =\ \prod(L_i\mid i\in I)\,\times\,\Free_{\varV}(I)$.
\end{minipage}\end{equation}

Suppose we choose, for each $i\in I$,
a lattice homomorphism $\gx_i\colon L_i\to L'$
which takes every $x\in L_i$ to an
element whose $i$-th coordinate is $x$ and whose
coordinate in $\Free_{\varV}(I)$ is the generator $g_i$.
\textup{(}For instance, such a family of
homomorphisms $\gx_i$ can be determined by fixing
an element $e_j$ in each $L_j$, and letting $\gx_i(x)$ have,
in addition to the two coordinates just specified,
$j$-th coordinate $e_j$ for all $j\in I-\{i\}$.\textup{)}

Finally, let $\gf'\colon L'\to M$ be the function
taking each pair $(x,w)$, where
\begin{equation}\begin{minipage}[c]{25pc}\label{d.x=}
$x\ =\ (x_i\mid i\in I)\in\prod(L_i\mid i\in I)$\quad
and\quad $w\in\Free_{\varV}(I)$,
\end{minipage}\end{equation}
to $\bar{w}(\gf_i(x)\mid i\in I)$,
where $\bar{w}\colon M^I\to M$ is the operation of evaluating
the lattice
expression $w\in\Free_{\varV}(I)$ at $I$-tuples of elements of $M$.

Then $\gf'$ is an isotone map, and for each $i\in I$,
$\gf'\gx_i=\gf_i$.
\end{lemma}

\begin{proof}[Sketch of proof]
Clear, except, perhaps, for the conclusion that $\gf'$ is isotone.

To get this, suppose
$(x,w)\leq(x',w')$ in $L'$.
Let us pass from $(x,w)$ to $(x',w')$ in finitely many steps.
At the first step, replace the coordinates $x_i$ of $x$
by $x'_i$ for all $i$ \emph{other} than the finitely many
values corresponding to the variables occurring in the term $w$.
This does not affect the image of our element under $\gf'$.
Next, one by one, replace the finitely many remaining
$x_i$ with the values $x'_i\geq x_i$.
The result is clearly greater than or equal to what we had.
Finally, replace $w$ by $w'\geq w$.
Again, the new value is greater than or equal to the old.
\end{proof}

We remark that the isotone map
$\gf'$ of the above lemma is not, in general, a~lattice homomorphism.
(For instance, let $I=\{0,1\}$, let $L_0=L_1=M=\SC 2$, let the
$\gf_i\colon L_i\to M$ be the identity map, and let $\varV$ be
any nontrivial lattice variety.
Denoting the generators of $\Free_\varV(I)$
by $g_0$ and $g_1$, we note that
\begin{equation}\begin{minipage}[c]{25pc}\label{d.neq}
$((0,1),g_1\mm g_2)\jj((1,0),g_1\mm g_2)\ =\ ((1,1),g_1\mm g_2)$
\quad in $L'$.
\end{minipage}\end{equation}
However, $\gf'$ maps each of the joinands on the left to $0$, but the
term on the right to~$1$.)

We have discussed symmetries of our construction informally above,
leaving it to the reader to see that a construction with a certain
sort of symmetry would imply corresponding properties of the maps
constructed.
These considerations can, of course, be formalized.
If we describe our constructions
as functors on appropriate categories of systems of lattices,
isotone maps, and distinguished elements, then constructions with
various sorts of symmetry allow us to strengthen the conclusion of
Theorem~\ref{T.main} to say that we have functors
respecting certain additional structure on those categories.
We shall not go into these details here, however.

We turn now to some variants of our main result.

\section{When does the same result hold for the inclusion\\
of a general partial lattice $P$ in its free
lattice $L$?}\label{S.complete}

If the lattice $M$ of Theorem~\ref{T.main}
happens to be a \emph{complete} lattice, the conclusion of
that theorem follows from a much more general fact:
Any isotone map from an order $P$ into a complete lattice
can be extended to any order extension $Q$ of $P$.
In other words, in the category of orders, complete
lattices are \emph{injective} with respect to inclusions of orders.

The inclusions of orders are not, up to isomorphism, the only
monomorphisms in the category of orders and isotone maps.
B.\ Banaschewski and G.\ Bruns \cite{B+B} characterize the inclusions
category-theoretically among the monomorphisms, calling them
the \emph{strict} monomorphisms, and they formulate the above result as
the statement that every complete lattice (in their terminology,
every complete partially ordered set) is a ``strict injective'';
to which they also prove the
converse \cite[Proposition~1, (i)$\!\iff\!$(ii)]{B+B}.

Theorem~\ref{T.main} can thus be looked at as saying that if
$P$ is the disjoint union of a family of lattices $L_i$ belonging
to a variety $\varV$, regarded as a partial lattice, then the
inclusion of $P$ in its free lattice $L=\Free_{\varV} P$
behaves a little better than a general
inclusion of orders, in that isotone maps of $P$ to arbitrary lattices,
and not only to complete lattices, can be extended to $L$.

In contrast, we shall see below
that the inclusion of a general partial lattice $P$
in its free lattice $\Free P$ behaves no better in this
way than do arbitrary extensions of orders, at least insofar as
isotone maps to lattices are concerned.
(For the concepts of a partial lattice and of the
free lattice on such an object, see \cite[Section~I.5]{GLT2},
\cite[Sections~I.5.4-I.5.5]{GLT3}.)

We begin with the building blocks from which the ``test cases''
showing this will be put together.

\begin{lemma}\label{L.Bool}
Let $B$ be a boolean lattice with $>2$ elements.
Then as an extension of $B-\{0,1\}$, the lattice
$\Free(B-\{0,1\})$ is isomorphic to $B$.
\end{lemma}

\begin{proof}
Let $P=B-\{0,1\}$.
The only joins that $P$ is missing are those that in $B$ yield $1$;
likewise, the only missing meets are those that yield $0$.
We shall show that all pairs of elements which had join $1$ in $B$
give equal joins in any lattice $L$ into which we map $P$
by a homomorphism of partial lattices.
By symmetry, the dual statement holds for $0$ and meets.
Hence the free lattice on $P$ just restores these two elements,
i.e., it is naturally isomorphic to~$B$.

So suppose we are given a map of $P$ into a lattice $L$, which
preserves the meets and joins of $P$.
By abuse of notation, we shall use the same symbols for elements
of~$P$ and their images in $L$.

Let us first consider any two elements $a,b\in P$ which are
distinct in $P$ from each other and from each other's complements, and
compare the joins $a\jj a\cp$ and $b\jj b\cp$ in $L$
(writing $(\ )\cp$ for complementation in $B$).

Note that in $B$, we have
\begin{equation}\begin{minipage}[c]{25pc}\label{d.a=vee}
$a\ =\ (a\mm b)\jj(a\mm b\cp)$\quad
and\quad $a\cp\ =\ (a\cp\mm b)\jj(a\cp\mm b\cp)$.
\end{minipage}\end{equation}
If the four meets appearing in these two expressions are all nonzero,
then they belong to $P$, and the relations~(\ref{d.a=vee}) hold there,
and hence in $L$.
In this situation, if we expand $a\jj a\cp$ in $L$
using these two formulas, we can rearrange the result as
$((a\mm b)\jj(a\cp\mm b))\jj
((a\mm b\cp)\jj(a\cp\mm b\cp))$, which (by~(\ref{d.a=vee})
with $a$ and $b$ interchanged) simplifies
to $b\jj b\cp$, giving the desired equality.
On the other hand, if any of the four pairwise meets
of $a$ and~$a\cp$ with $b$ and~$b\cp$
is zero, this can, under the assumptions made in the
preceding paragraph, be true only of
one such meet; say $a\mm b=0$.
Then we can repeat the above computation, everywhere
omitting ``$(a\mm b)\jj$''.
(Thus, we have a version of~(\ref{d.a=vee}) with the
first equation simplified to $a=a\mm b\cp$, and the second unchanged.)
With this slight modification, our computation still works, and
again gives $a\jj a\cp=b\jj b\cp$.

So let us write $i$ for the common value,
for all $a\in P$, of $a\jj a\cp\in L$.
We now consider two elements $a,b\in P$
which are not assumed to be complements, but whose join in $B$ is $1$.
This relation implies that $b\geq a\cp$; note that both these
terms lie in $P$.
Hence in $L$ we have $a\jj b\ \geq\ a\jj a\cp\ =\ i$,
while the reverse inequality holds because $a\leq i,\ b\leq i$.
Thus, $a\jj b=i$, completing the proof that
all pairs of elements having join $1$ in $B$
have the same join, namely $i$, in~$L$.
\end{proof}

Let us now consider, independent of the above
result, the same inclusion $B-\nolinebreak\{0,1\}\ci B$
in the context of isotone maps.

\begin{lemma}\label{L.isot_on_freeB}
Let $M$ be a lattice, $X$ a subset of $M$, and $B$ the free Boolean
lattice on~$X$.
Then there exists an isotone map $\gf\colon B-\{0,1\}\to M$ with the
property that
\begin{equation}\begin{minipage}[c]{25pc}\label{d.bounds}
The pairs of elements $y,z\in M$ such that $\gf$ can be
extended to an isotone map $\bar{\gf}\colon  B\to M$
taking $0$ to $y$ and $1$ to $z$, are precisely those
for which $y$ is a lower bound, and $z$ an upper bound, for $X$ in $M$.
\end{minipage}\end{equation}
\end{lemma}

\begin{proof}
Let $B$ have free generators $g_x$ for $x\in X$.
For every $a\in B-\{0,1\}$, let $\gf(a)$ be the join
in $M$ of all elements of the form
\begin{equation}\begin{minipage}[c]{25pc}\label{d.x1,xn}
$x_1\mm\cdots\mm x_n$
\end{minipage}\end{equation}
where $n\geq 1$, and $x_1,\dots,x_n$ are distinct elements of $X$
such that for some choice of
$\ge_1,\dots,\ge_n\in\{1,\mathrm{c}\}$, we have
\begin{equation}\begin{minipage}[c]{25pc}\label{d.ageq}
$a\ \geq\ g_{x_1}^{\ge_1}\mm\cdots\mm
g_{x_n}^{\ge_n}$.
\end{minipage}\end{equation}
Here for $b\in B$, $b^\ge$ denotes $b$ if $\ge=1$,
the complement of $b$ if $\ge=\mathrm{c}$.

Because $a\neq 0$, the set of instances of~(\ref{d.ageq})
is nonempty, hence so is the set of joinands~(\ref{d.x1,xn}).
These sets are in general infinite; however, if we take the least
subset $X_0\ci X$ such that $a$ is in the Boolean sublattice
generated by the $g_x$ with $x\in X_0$, then $X_0$ is finite
and (since $a\neq 0,1$) nonempty; and we find that
the irredundant relations~(\ref{d.ageq}) (those relations~(\ref{d.ageq})
from which no meetand can be dropped)
involve only terms $g_x^\ge$ with $x\in X_0$.
Thus, each expression~(\ref{d.x1,xn}) in our description of~$\gf(a)$
is majorized by one that
arises from one of these finitely many irredundant
relations~(\ref{d.ageq}); so the join describing
$\gf(a)$ is effectively a finite join, and so exists in~$M$.

It is not hard to see from our
definition that $\gf$ is isotone, and that
for all $x\in X$, $\gf(g_x)=\gf(g_x\cp)=x$.

Suppose now that we have an extension $\bar{\gf}\colon B\to M$ of
this isotone map $\gf$.
Then for every $x\in X$,
$\bar{\gf}(1)\geq\gf(g_x)=x$, so $\bar{\gf}(1)$ is an
upper bound of $X$.
Conversely, any upper
bound for $X$ in $M$ will majorize all elements~(\ref{d.x1,xn}),
and hence all joins of such elements,
hence will indeed be an acceptable choice
for a value of $\bar{\gf}(1)$
making $\bar{\gf}$ isotone.
Though our construction of $\gf$ is not symmetric in $\jj$ and
$\mm$, the duals of these
observations are easily seen to hold, so the
choices for $\bar{\gf}(0)$ are, likewise, the lower bounds of~$X$.
\end{proof}

Note that (\ref{d.bounds}) above can be summarized as saying that
\begin{equation}\begin{minipage}[c]{25pc}\label{d.bounds_alt}
The upper and lower bounds in $M$ of $\gf(B-\{0,1\})$ are
the same as the upper and lower bounds in $M$ of $X$.
\end{minipage}\end{equation}

In our next result, for any two partial lattices
$P$ and $Q$, we will denote by $P+Q$ the disjoint union of $P$
and $Q$, made a partial lattice using the partial meet and join
operations of $P$ and $Q$, together with the further
meet and join relations corresponding to the condition
that every element of $P$ be majorized by every element
of $Q$ (namely, $p\mm q=p$ and $p\jj q=q$ for all $p\in P,\ q\in Q)$.
It is not hard to see that
\begin{equation}\begin{minipage}[c]{25pc}\label{d.FreeP+Q}
$\Free(P+Q)\ \cong\ \Free P+\Free Q$.
\end{minipage}\end{equation}

\begin{theorem}\label{T.complete}
Let $M$ be a lattice.
Then the following conditions are equivalent.
\vspace{.3em}
\begin{equation}\begin{minipage}[c]{25pc}\label{d.complete}
$M$ is complete.
\end{minipage}\end{equation}
\begin{equation}\begin{minipage}[c]{25pc}\label{d.free_extend}
Any isotone map from a partial lattice $P$ to $M$
can be extended to an isotone map $\Free P\to M$.
\end{minipage}\end{equation}
\begin{equation}\begin{minipage}[c]{25pc}\label{d.Boole_extend}
For $B$ a free Boolean lattice on a nonempty
set, any isotone map $B_1-\{0,1\}\to M$ can be extended
to an isotone map $B\to M$; and
for $B_1$, $B_2$ any two free Boolean lattices on nonempty
sets, any isotone map $(B_1-\{0,1\})+(B_2-\{0,1\})\to M$
can be extended to an isotone map $B_1+B_2\to M$.
\end{minipage}\end{equation}
\end{theorem}

\begin{proof}
(\ref{d.complete})$\!\implies\!$(\ref{d.free_extend})
is a case of \cite[Proposition~1, (i)$\!\implies\!$(ii)]{B+B},
which says that every complete lattice is injective with
respect to inclusions of orders.
In view of Lemma~\ref{L.Bool} and~(\ref{d.FreeP+Q}), the implication
(\ref{d.free_extend})$\!\implies\!$(\ref{d.Boole_extend}) is clear.

To complete the argument, assume~(\ref{d.Boole_extend}).

Calling on the first statement of~(\ref{d.Boole_extend}),
together with the case $X=M$ of the preceding lemma,
we see that $M$ must have a greatest and a least element.

Now take any nonempty subset $X_1\ci M$, let $X_2$ be the set of
its upper bounds (which is nonempty, since $M$ has a greatest
element), let $B_1$ be the free Boolean lattice on $X_1$, and let
$B_2$ be the free Boolean lattice on $X_2$.
Map $B_1-\{0,1\}$ and~$B_2-\{0,1\}$ into $M$
by maps $\gf_1$, $\gf_2$ satisfying~(\ref{d.bounds})
with respect to $X_1$ and $X_2$, respectively.
By the equivalence of~(\ref{d.bounds}) and~(\ref{d.bounds_alt}),
$\gf_1(B_1-\{0,1\})$ is
majorized by all upper bounds of $X_1$, i.e., by all elements of $X_2$,
hence (again using that equivalence) by
all elements of $\gf_2(B_2-\{0,1\})$;
so $\gf_1$ and $\gf_2$ together constitute
an isotone map $\gf\colon (B_1-\{0,1\})+(B_2-\{0,1\})\to M$.
Extending this to the free lattice $B_1+B_2$
on that partial lattice, we see that the image
of the $1$ of $B_1$ (and likewise that of the
$0$ of $B_2)$ will be both an upper bound of $X_1$ and a lower
bound of $X_2$, hence must be a least upper bound of $X_1$.
So $M$ is upper semicomplete.

By symmetry (or by the known fact that in a lattice with
$0$ and $1$, upper semicompleteness and
lower semicompleteness are equivalent), $M$ is also
lower semicomplete, establishing~(\ref{d.complete}).
\end{proof}

\section{Lattices amalgamated over convex retracts}\label{S.retracts}

The results of the preceding section show that Theorem~\ref{T.main},
looked at as a property of the inclusion of
a certain kind of partial lattice $P$ in $\Free_{\varV} P$,
does not go over to the inclusion of
a general partial lattice $P$ in its free lattice.
Can we describe other interesting partial lattices $P$
for which it does?

In proving Theorem~\ref{T.main}, after reducing to the
case where the given lattices contained elements $e_i$ that
mapped to the same element of $M$, we effectively proved that the
free lattice on the union of those lattices with
amalgamation of the $e_i$ had the desired extension property.
The next theorem will slightly generalize this result, replacing
the singletons $\{e_i\}$ with any family of isomorphic sublattices that
are both retracts of the $L_i$, and convex therein.
We will need the following observation.

\begin{lemma}\label{L.cvx_retr}
Let $M$ be a lattice, and $\gr$ a lattice-theoretic
retraction of $M$ to a convex sublattice.
Then if an element $x\in M$ is majorized by some
element of $\gr(M)$, then it is majorized by $\gr(x)$.
\end{lemma}

\begin{proof}
Say $x\leq r\in\gr(M)$.
Applying $\gr$ to this relation, and taking the join
of the original relation with the resulting one, we get
$x\jj\gr(x)\leq r$.
Hence $x\jj\gr(x)$ lies in the interval between $\gr(x)$
and $r$, so as $\gr(M)$ is assumed convex, $x\jj\gr(x)\in\gr(M)$.
This means that $x\jj\gr(x)$ is fixed under the idempotent
lattice homomorphism $\gr$; but its image under that map is
$\gr(x)\jj\gr(x)=\gr(x)$.
Thus, $x\jj\gr(x)=\gr(x)$, which is equivalent to the desired
conclusion $x\leq\gr(x)$.
\end{proof}

In the above lemma, the assumption that $\gr$ is a lattice
homomorphism could have been weakened to say that it is
a join-semilattice homomorphism.
We have stated it as above for conceptual simplicity, and because
in the proof of the next result,
the maps $\gr_i$ must be lattice homomorphisms anyway.

\begin{theorem}\label{T.retract}
Let $(L_i \mid i\in I)$ be a family of lattices which are
disjoint except for a common sublattice $K$, which is convex
in each $L_i$, and is a retract of each $L_i$ via a lattice-theoretic
retraction $\gr_i\colon L_i\to K$.

Let $P$ denote the partial lattice given by the union of
the $L_i$ with amalgamation of the common sublattice $K$,
and let $L=\Free_\varV P$,
where $\varV$ is any variety containing all the $L_i$.

Then for any lattice $M$ \textup{(}not necessarily
belonging to $\varV$\textup{)} given with isotone maps
$\gf_i\colon L_i\to M$ agreeing on $K$, there exists an isotone map
$\gf\colon L\to M$ extending all the $\gf_i$.

In other words, every isotone map $P\to M$ extends to $L$.
\end{theorem}

\begin{proof}
Let us assume that $I$ does not contain the symbol $0$,
and use $0$ to index the factor $K$
in $K\times(\prod L_i\mid i\in I)$.
Now let $L'$ denote the sublattice of that direct product consisting
of those elements $f$ such that
$f(i)=f(0)$ for almost all $i$, and $\gr_i(f(i))=f(0)$ for all $i$.
Then we can map each $L_i$ into $L'$ by sending $x\in L_i$
to the element having $i$-th coordinate $x$, and having
$\gr_i(x)$ for all other coordinates (including the $0$-th coordinate).
These maps are lattice homomorphisms (this is where we need
the $\gr_i$ to be lattice homomorphisms and not just
join-semilattice homomorphisms), which agree on
$K$; hence they extend to a lattice homomorphism $L\to L'$.

We shall now map $L'$ isotonely to $M$ using the
idea of Lemma~\ref{L.M'toM}.
Namely, given $f\in L'$, we define
\begin{equation}\begin{minipage}[c]{25pc}\label{d.gy_again}
$\gy(f)\ =\ \begin{cases}
\MMm{\gf_i(f(i))}{i\in I} &
\mbox{if for all $i\in I$, $f(i)\leq f(0)$,}\\[.3em]
\JJm{\gf_i(f(i))}{i\in I,\,f(i)\not\leq f(0)} &
\mbox{otherwise.}
\end{cases}$
\end{minipage}\end{equation}
These are defined because for each $f$, all but finitely
many $i\in I$ have $\gf_i(f(i))$ equal to the image of $f(0)$
in $M$.
(Recall that $f(0)\in K$, and all $\gf_i$ agree on $K.)$
We now claim that
\begin{equation}\begin{minipage}[c]{25pc}\label{d.isotone_again}
$\gy$ is isotone,
\end{minipage}\end{equation}
and
\begin{equation}\begin{minipage}[c]{25pc}\label{d.all_but_one_again}
for every $i\in I$, and every $f\in L'$ such that $f(j)=f(0)$ for
all $j\neq i$, we have $\gy(f) = \gf_i(f(i))$.
\end{minipage}\end{equation}

Assertion~(\ref{d.all_but_one_again}) is clear.
The proof
of~(\ref{d.isotone_again}) is exactly like that of the corresponding
statement,~(\ref{d.isotone}), in the proof of Lemma~\ref{L.M'toM},
once we know the analog of~(\ref{d.subset}), namely
\begin{equation}\begin{minipage}[c]{25pc}\label{d.subset_again}
for $f\leq g$ in $L'$, we have
$\{i\mid f(i)\not\leq f(0)\}\ \ci\ \{i\mid g(i)\not\leq g(0)\}$.
\end{minipage}\end{equation}
To prove~(\ref{d.subset_again}), consider any $i$ not lying in the
right-hand side.
Then
\begin{equation}\begin{minipage}[c]{25pc}\label{d.fi_<gi_<g0}
$f(i)\ \leq\ g(i)\ \leq\ g(0)\in\gr(M)$,
\end{minipage}\end{equation}
so by Lemma~\ref{L.cvx_retr},
$f(i)\leq\gr(f(i))=f(0)$, showing that $i$ also fails to lie in the
left-hand set.

Composing $\gy$ with the map $L\to L'$ of the first
paragraph of this proof, we get our desired isotone map $L\to M$.
\end{proof}

(If we think of the constant $e$ of Lemma~\ref{L.M'toM} as
``sea level'', then the $f(0)$ of the above proof
brings in ``tides''.)

We remark that though in the free-lattice-with-amalgamation
$L$ of the above proof, $K$ is necessarily
a retract, since it was a retract in each of
the $L_i$, it does not follow similarly that $K$ is convex in $L$.
To see this, let us first note an example of a lattice $L'$
having a sublattice $K$ which is convex and a retract in each
of two sublattices $L_0$ and $L_1$ containing $K$, but is not
convex in the sublattice that these generate.
Let $L'$ be the lattice of all subspaces of a $3$-dimensional
vector space $V$ over any field, let $K=\{\{0\},a\}$ where $a$ is
a $2$-dimensional subspace of $V$, and
let each $L_i$ be the sublattice generated by $a$ and a
$1$-dimensional subspace $b_i$ not contained in $a$, with $b_0\neq b_1$.
Then the stated hypotheses are satisfied, but
$0<(b_0\jj b_1)\mm a<a$, so $K$ is not
convex in the lattice generated by $L_1$ and $L_2$.
It easily follows that in the free product $L$ of $L_0$ and $L_1$
with amalgamation of $K=\{0,1\}$, we likewise have
$0<(b_0\jj b_1)\mm a<a$ with the middle term not in $K$.

\section{Semilattice variants---two easy results}\label{S.semilat1}
In our main theorem, free products of
lattices $L_i$, whose normal role is to admit a lattice
homomorphism extending a given family of lattice homomorphisms
on the~$L_i$, were made to do the same for isotone maps
(homomorphisms of orders).
One might expect it to be easier to get similar
results if the gap between lattices and orders is replaced by
one of the smaller gaps between lattices and semilattices,
or between semilattices and orders.

For the latter case, the result is indeed easy; it is only
for parallelism with our other results that
we dignify it with the title of theorem.

\begin{theorem}\label{T.semilat_to_order}
Let $(L_i \mid i\in I)$ be a family of
join-semilattices, and $\gf_i\colon L_i\to M$ a family of
isotone maps from the $L_i$ to a common join-semilattice.
Let $L$ denote the free product of the $L_i$ as join-semilattices.
Then there exists an isotone map $\gf\colon L\to M$
whose restrictions to the $L_i\ci L$ are the $\gf_i$.
\end{theorem}

\begin{proof}
The general element $x\in L$ is a formal
join $x_{i_1}\jj\cdots\jj x_{i_n}$
of elements $x_{i_m}\in L_{i_m}$, where $i_1,\dots,i_n$
are a finite nonempty family of distinct indices in $I$.
If we send each such $x$ to
$\gf_{i_1}(x_{i_1})\jj\cdots\jj\gf_{i_n}(x_{i_n})$,
this is easily seen to have the desired properties.
\end{proof}

There was no analog, in the above result, to the $\varV$ of
Theorem~\ref{T.main}, since the variety of semilattices has
no proper nontrivial subvarieties.

On the other hand, if we wish to get an analog of
Theorem~\ref{T.main} with the $L_i$ and~$M$ again lattices, but for
semilattice homomorphisms, rather than isotone maps, we may
again start with lattices $L_i$ in an arbitrary lattice variety $\varV$.
For this situation the authors have not been able to prove
the full analog of Theorem~\ref{T.main}.
The difficulty with adapting our proofs of that theorem
is that the map $\gy$ of Lemma~\ref{L.M'toM}, though isotone,
does not respect joins; nor do the variant constructions
of Section~\ref{S.alt_pfs}.

The map of Lemma~\ref{L.M'toM} does, however, respect joins when
the $e_i$ are least elements in the $\bar{L}_i$.
In that case, the composite
$L'\to M'\to M$ reduces to the map~(\ref{d.pre_gy}) in our sketch of
the ``easy case'' of Theorem~\ref{T.main}, and we
find that if the $\gf_i$ are join-semilattice homomorphisms, that
composite will also be one.
Hence we get

\begin{proposition}\label{P.semilat_bdd_below}
Let $\varV$ be a nontrivial variety of lattices,
$\E L=(L_i \mid i\in I)$ a~family of lattices in $\varV$,
$L=\Free_{\varV}\E L$, and $\gf_i\colon L_i\to M$ a family of
join-semilattice homomorphisms from the $L_i$ to a common lattice,
not necessarily belonging to $\varV$.

If the image-sets $\gf_i(L_i)$ have a common lower bound $e\in M$, then
there exists a join-semilattice homomorphism $\gf\colon L\to M$
whose restrictions to the $L_i\ci L$ are the~$\gf_i$.
In particular, this is so if $M$ has a least element,
or if $I$ is finite and every~$L_i$ has a least element.\qed
\end{proposition}

One could modify this result in the spirit of Theorem~\ref{T.retract},
assuming that each $L_i$ has a retraction $\gr_i$ to a common
\emph{ideal} $K$ on which the $\gf_i$ agree.

In another direction, the condition in
Proposition~\ref{P.semilat_bdd_below} that there
exist a common lower bound $e$ in $M$ to all the $\gf_i(L_i)$ can
be weakened slightly
(for $I$ infinite) to say that $M$ contains a chain $C$ such that
every $\gf_i(L_i)$ is bounded below by some member of $C$.
Let us sketch the argument that gets this, by
transfinite induction, from the statement as given.
First, by passing to a subchain, assume
without loss of generality that $C$ is dually well-ordered.
Then apply Proposition~\ref{P.semilat_bdd_below}, first,
to those $L_i$ such that $\gf_i(L_i)$ is bounded below
by the top element, $c_0$, of $C$, concluding that those
$\gf_i$ can be factored through some lattice $L'_{(0)}$ in $\varV$.
Then go to the next member, $c_1$, of $C$,
and combine $L'_{(0)}$ with all the $L_i$ that are bounded
below by $c_1$ but not by $c_0$, factoring
these together through a lattice $L'_{(1)}\in\varV$; and so on.
As in the discussion following~(\ref{d.L'k=}), we take the
union of the preceding steps whenever we hit a limit ordinal.

\section{Semilattice variants---a harder result}\label{S.semilat2}

What if we have nothing like the lower-bound condition
of Proposition~\ref{P.semilat_bdd_below}?

For free products taken in the variety $\varL$ of all lattices,
the analog of that proposition,
without the lower bound condition, is obtained
in~\cite[middle of p.~239, ``We note finally\,...'']{GLP70}.
Indeed, the map $f$ used in~\cite{GLP70} to prove Theorem~\ref{T.main}
for $L=\Free\E L$ has the property that
$f(x\jj y)=f(x)\jj f(y)$ \emph{except} possibly
when $x$ and $y$ are bounded below by elements
$x_{(i)},\,y_{(i)}\in L_i$ for some $i$, and
$\gf_i(x_{(i)}\jj y_{(i)})>\gf_i(x_{(i)})\jj
\gf_i(y_{(i)})$.
(Cf.\ \cite[p.238, (ii)]{GLP70}.)
If the $\gf_i$ are join-semilattice homomorphisms, that
strict inequality never occurs, so $\gf$ is
also a join-semilattice homomorphism.

If $\varV$ is a nontrivial variety of lattices containing the $L_i$,
we do not know whether we can get the corresponding result for
the free product of the $L_i$ in $\varV$, but we shall show below that
we can get such a result for their
free product in the larger class $\varD\circ\varV$
(definition recalled in~(\ref{d.circ}) and~(\ref{d.D}) below).
Our construction will be similar in broad outline to those used in
preceding sections, but the intermediate lattice $L'$, rather than
being a subdirect product, will be a certain lattice of downsets
in a direct product.

We recall the definition:
\begin{equation}\begin{minipage}[c]{25pc}\label{d.circ}
If $\mathbf{K}_1$ and $\mathbf{K}_2$ are classes
of lattices, then the class of those lattices $L$ which
admit homomorphisms
$\ge\colon L\to L_2$ such that $L_2\in\mathbf{K}_2$, and such that
the inverse image of every element of $L_2$
lies in $\mathbf{K}_1$, is denoted $\mathbf{K}_1\circ\mathbf{K}_2$.
\end{minipage}\end{equation}
The class $\mathbf{K}_1\circ\mathbf{K}_2$ is often called the
\emph{product} of the classes $\mathbf{K}_1$ and $\mathbf{K}_2$,
but we will not use that name here, to avoid confusion with
direct products and free products of lattices.

If $\mathbf{K}_1$ and $\mathbf{K}_2$ are varieties, the class
$\mathbf{K}_1\circ\mathbf{K}_2$ need not be a variety; but
as noted in~\cite{AIM},
if $\mathbf{K}_1$ and $\mathbf{K}_2$ are prevarieties or quasivarieties
(classes closed under taking direct products and sublattices;
respectively, under taking direct products,
ultraproducts, and sublattices),
then $\mathbf{K}_1\circ\mathbf{K}_2$ will also be a
prevariety, respectively a quasivariety.
In~particular, if $\mathbf{K}_1$ and $\mathbf{K}_2$ are varieties,
$\mathbf{K}_1\circ\mathbf{K}_2$ is, at least, a quasivariety.

We also recall the standard notation:
\begin{equation}\begin{minipage}[c]{25pc}\label{d.D}
The variety of distributive lattices is denoted $\varD$.
\end{minipage}\end{equation}

Now suppose $(L_i \mid i\in I)$ is a family of lattices.
To begin the construction of the lattice $L'$ that we shall use
in proving our final result,
let us adjoin to each $L_i$ a new top element, $1_i$,
form the direct product $\prod(L_i+\{1_i\})$, and define the subset
\begin{equation}\begin{minipage}[c]{25pc}\label{d.P}
$P\ =\ \{f\in\prod(L_i+\{1_i\})\mid\{i\mid f(i)\neq 1_i\}$\  is
finite but nonempty$\}$.
\end{minipage}\end{equation}
The condition that $\{i\mid f(i)\neq 1_i\}$ be
nonempty means that we are excluding the top element
of $\prod(L_i+\{1_i\})$; hence if $|I|>1$, $P$ is not a lattice,
though it is a lower semilattice.

For each $i\in I$, let us define a map $\gq_i\colon  L_i\to P$ by
\begin{equation}\begin{minipage}[c]{25pc}\label{d.theta}
$\gq_i(x)(i)=x,\qquad\gq_i(x)(j)=1_j$\ \ for $j\neq i$.
\end{minipage}\end{equation}
We see that every element of $p\in P$ has a representation
\begin{equation}\begin{minipage}[c]{25pc}\label{d.finite_meet}
$p\ =\ \gq_{i_1}(x_1)\mm\cdots\mm\gq_{i_n}(x_n)$\quad with $n>0$,
\end{minipage}\end{equation}
unique up to order of terms, where $i_1,\dots,i_n$ are
distinct elements of $I$, and $x_m\in\nolinebreak L_{i_m}$.

We now let
\begin{equation}\begin{minipage}[c]{25pc}\label{d.L'}
$L'\ =$ the set of all nonempty finitely generated downsets
$F\ci P$ such that
\end{minipage}\end{equation}
\begin{equation}\begin{minipage}[c]{25pc}\label{d.x_vee_y}
for all $i\in I$ and $x,\,y\in L_i$,
if $\gq_i(x),\,\gq_i(y)\in F$, then $\gq_i(x\jj y)\in F$.
\end{minipage}\end{equation}
Thus
\begin{equation}\begin{minipage}[c]{25pc}\label{d.elts_of_L'}
Each element $F\in L'$ is the union of the principal downsets
\mbox{$\downarrow(\gq_{i_1}(x_1)\mm\dots\mm\gq_{i_n}(x_n))$}
determined by its finitely many maximal elements
$\gq_{i_1}(x_1)\mm\dots\mm\gq_{i_n}(x_n)$.
Moreover, for each $i$, $F$ can have at most one such
maximal element of the form $\gq_i(x)$
(i.e., with $n=1$, and with the one meetand arising from $L_i$).
\end{minipage}\end{equation}

The last sentence above follows from~(\ref{d.x_vee_y}):
Given distinct $\gq_i(x),\,\gq_i(y)\in F$,
we also have $\gq_i(x\jj y)\in F$, so $\gq_i(x)$ and
$\gq_i(y)$ cannot both be maximal in $F$.

Let us now prove

\begin{lemma}\label{L.DoV}
Let $(L_i \mid i\in I)$ be a family of lattices, and let
$L'$ be constructed as in\textup{~(\ref{d.P})-(\ref{d.x_vee_y})}
above.
Then
\begin{equation}\begin{minipage}[c]{25pc}\label{d.L'_is_lattice}
$L'$, partially ordered by inclusion, is a lattice.
\end{minipage}\end{equation}
\begin{equation}\begin{minipage}[c]{25pc}\label{d.xi_i_is_hom}
For each $i\in I$, the map $\gx_i\colon L_i\to L'$
defined by $\gx_i(x)=\,\downarrow\gq_i(x)$
is a lattice homomorphism.
\end{minipage}\end{equation}
\begin{equation}\begin{minipage}[c]{25pc}\label{d.DoV}
If $\varV$ is any prevariety containing all the $L_i$,
then $L'\in\varD\circ\varV$.
\end{minipage}\end{equation}
\end{lemma}

\begin{proof}
In verifying~(\ref{d.L'_is_lattice}),
the only points that need a moment's thought are (i) that
the intersection $F\cap G$ of two sets as in~(\ref{d.elts_of_L'})
remains nonempty and finitely generated; but indeed,
in any meet-semilattice, the intersection of two nonempty finitely
generated downsets $\bigcup\downarrow p_i$
and $\bigcup\downarrow q_j$ is the nonempty finitely generated
downset $\bigcup\downarrow p_i\mm q_j$;
(ii) that the closure operation of~(\ref{d.x_vee_y}) cannot
produce the element $(1_i)_{i\in I}\notin P$; this follows from
the fact that each $L_i$ is closed under joins in $L_i+\{1_i\}$;
and (iii) that repeated application of that operation
when we form a join $F\jj G$ cannot
lead to a violation of finite generation as a downset.
This is clear once we observe that in constructing
$F\jj G$ from $F\cup G$, it is enough to apply
the closure operation of~(\ref{d.x_vee_y}) to pairs consisting of one
of the finitely many maximal elements of $F$ and one
of the finitely many maximal elements of $G$ (and then close
as a downset).

Statement~(\ref{d.xi_i_is_hom}) is easily checked.
(Here~(\ref{d.x_vee_y}) guarantees that $\gx_i$ respects joins ---
that is the point of that condition.)

To show~(\ref{d.DoV}), let us now adjoin to each $L_i$ a bottom element
$0_i$, and define maps $\gp_i\colon L'\to \{0_i\}+L_i$ as follows.
For $F\in L'$,
\begin{equation}\begin{minipage}[c]{25pc}\label{d.pi_i}
If there are elements $x\in L_i$ such that $\gq_i(x)\in F$,
let $\gp_i(F)$ be the largest such $x$
(cf.\ second sentence of~(\ref{d.elts_of_L'})).\\[0.5em]
If there are no such $x$, let $\gp_i(F)\ =\ 0_i$.
\end{minipage}\end{equation}
In view of~(\ref{d.x_vee_y}), each $\gp_i$
is a homomorphism; hence together they give us a homomorphism
$\gp\colon L'\to\prod(\{0_i\}+L_i\mid i\in I)\in\varV$.

We claim that the inverse image under $\gp$ of each
$f\in\prod(\{0_i\}+L_i\mid i\in I)$ is distributive.
Indeed, when we take the join of two elements $F,\,G\in\gp^{-1}(f)$,
we see that for each $i$, the sets $F$ and $G$ agree in what
elements $\gq_i(x)$ they contain, hence there is no occasion
for enlarging $F\cup G$ via~(\ref{d.x_vee_y}).
So $F\jj G=F\cup G$.
We always have $F\mm G=F\cap G$
in $L'$; hence $\gp^{-1}(f)$ is a lattice of subsets of $P$
under unions and intersections, hence it is distributive.
Thus, $L'\in\varD\circ\varV$, as claimed.
\end{proof}

Now suppose that for each $i\in I$ we are given an upper semilattice
homomorphism $\gf_i\colon L_i\to M$, for a fixed lattice $M$.
We define $\gy\colon L'\to M$ by
\begin{equation}\begin{minipage}[c]{25pc}\label{d.gy_new}
$\gy(F)\ =
\ \JJm{\gf_{i_1}(x_1)\mm\dots\mm\gf_{i_n}(x_n)\ }%
{\ \gq_{i_1}(x_1)\mm\dots\mm\gq_{i_n}(x_n)\in F}$.
\end{minipage}\end{equation}
This is formally an infinite join; but it is clearly equivalent to
the corresponding join over the finitely many maximal elements of $F$,
hence is defined.

We claim that
\begin{equation}\begin{minipage}[c]{25pc}\label{d.join-hom}
$\gy$ is a join-semilattice homomorphism.
\end{minipage}\end{equation}
To see this, note that if we temporarily extend
the definition~(\ref{d.gy_new}) to arbitrary
finitely generated downsets $F$, not necessarily
satisfying~(\ref{d.x_vee_y}), then we have
\begin{equation}\begin{minipage}[c]{25pc}\label{d.cup=jj}
$\gy(F\cup G)\ =\ \gy(F)\jj\gy(G)$.
\end{minipage}\end{equation}
Now for $F,\,G\in L'$, the element $F\jj G$ is obtained
by bringing into $F\cup G$ elements $\gq_i(x\jj y)$
where $\gq_i(x)\in F$ and $\gq_i(y)\in G$
(and the elements they majorize).
In this situation, the join defining $\gy(F\cup G)$ already
contains joinands $\gf_i(x)$ and $\gf_i(y)$,
resulting from the presence of $\gq_i(x)$ and $\gq_i(y)$
in $F$ and $G$, hence its value in $M$ already majorizes
$\gf_i(x)\jj\gf_i(y) =\gf_i(x\jj y)$.
So bringing $\gq_i(x\jj y)$ into $F\cup G$ does not increase
its image under $\gy$, establishing~(\ref{d.join-hom}).

Finally, comparing the definition~(\ref{d.xi_i_is_hom}) of the
$\gx_i$ and the definition~(\ref{d.gy_new}) of $\gy$, we see that
\begin{equation}\begin{minipage}[c]{25pc}\label{d.gfgy_i}
For all $i\in I,\quad\gf_i\ =\ \gy\,\gx_i$.
\end{minipage}\end{equation}

Now given $\varV$ as in~(\ref{d.DoV}),
let $\E L=(L_i\mid i\in I)$ and
$L=\Free_{\varD\circ\varV}\E L$.
Then the lattice homomorphisms $\gx_i\colon  L_i\to L'$ are
equivalent to a single homomorphism $\gx\colon L\to L'$;
and we see that by taking
$\gf=\gy\,\gx\colon L\to L'\to M$, we get our desired result:

\begin{theorem}\label{T.semilat}
Let $\E L=(L_i \mid i\in I)$ be a family of
lattices, and $\gf_i\colon L_i\to M$ a family of join-semilattice
homomorphisms from the $L_i$ to a common lattice.
Suppose all $L_i$ lie in some prevariety $\varV$ of lattices,
and let $L=\Free_{\varD\circ\varV}\E L$.
Then there exists a join-semilattice homomorphism $\gf\colon L\to M$
whose restrictions to the $L_i$ are the~$\gf_i$.\qed
\end{theorem}

Let us show now by example that the lattice $L'$ constructed in
the above proof may fail to lie in $\varV$ itself.
We start with two distributive lattices, namely,
the one-element lattice $L_0=\{e\}$, and the four-element lattice
$L_1$ generated by two elements $a$ and $b$.
Let us use bar notation for the images of these generators
under the embeddings $\gx_i\colon  L_i\to L'$, so that
$\bar{e}=\gx_0(e)=\ \downarrow(e,1_1)$,
$\bar{a}=\gx_1(a)=\ \downarrow(1_0,a)$,
$\bar{b}=\gx_1(b)=\ \downarrow(1_0,b)$.
We claim that in $L'$,
\begin{equation}\begin{minipage}[c]{25pc}\label{d.nondist}
$\bar{e}\mm(\bar{a}\jj\bar{b})\ \neq
\ (\bar{e}\mm\bar{a})\jj(\bar{e}\mm\bar{b})$.
\end{minipage}\end{equation}
Indeed, one finds that the left-hand side
of~(\ref{d.nondist}) is the principal down-set $\downarrow(e,a\jj b)$,
while the right-hand side is $(\downarrow(e,a))\cup(\downarrow(e,b))$,
a nonprincipal down-set.
Hence $L'$ is not distributive.
It is not even modular: one can similarly verify that
a copy of $N_5$ is given by the elements
\begin{equation}\begin{minipage}[c]{25pc}\label{d.N_5}
$(\bar{e}\mm\bar{a})\jj(\bar{e}\mm
\bar{b})\jj(\bar{a}\mm\bar{b}),\quad
\bar{a}\jj(\bar{e}\mm\bar{b}),\quad
\bar{a}\jj(\bar{e}\mm(\bar{a}\jj\bar{b})),\quad
\bar{a}\jj\bar{b},\quad
(\bar{e}\mm\bar{a})\jj\bar{b}$.
\end{minipage}\end{equation}

Evidence suggesting that the task of extending
semilattice homomorphisms from a family of lattices to
their free product is likely to be harder than the corresponding
task for isotone maps is \cite[Theorem~1]{B+L}
$=$  \cite[Theorem 2.8]{H+K}, which
says that the injective objects in the category
of meet-semilattices are the \emph{frames}, i.e., the complete
lattices satisfying the join-infinite distributive identity.
(This result is generalized in \cite[Theorem~3.1]{Z+Z}.)
Thus, dually, the injective join-semilattices are the complete
lattices satisfying the meet-infinite distributive identity;
in particular, they are distributive; so Theorem~\ref{T.semilat} does
not ``almost'' follow from a general injectivity
statement, as Theorem~\ref{T.main} did.

(While on the topic of injective
objects, what are the injectives in the variety of lattices?
It is shown in \cite[next-to-last paragraph]{B+B} that
the only one is the trivial lattice.
This is generalized in \cite{AD} to any nontrivial
variety $\varV$ of lattices other than the
variety of distributive lattices, and in \cite{EN}, with a
very quick proof, to any class of lattices containing
a $3$-element chain and a nondistributive lattice.)

\section{Questions.}\label{S.questions}
The example following Theorem~\ref{T.semilat} does
not mean that there is no way to factor a family of maps
as in that theorem through the free product of the $L_i$ in
$\varV$; only that the construction by which we have proved
that theorem doesn't lead to such a factorization.
Indeed, for that particular pair of lattices,
one does have such a factorization, by the final clause of
Proposition~\ref{P.semilat_bdd_below}.
So we ask

\begin{question}\label{Q.semilat}
For $\varV$ a general nontrivial variety of lattices,
can one prove the full analog of Theorem~\ref{T.main}
with join-semilattice homomorphisms in place of isotone maps
\textup{(}i.e., a result like Theorem~\ref{T.semilat} with
$\varV$ in place of $\varD\circ\varV$; equivalently, a result
like Proposition~\ref{P.semilat_bdd_below} without the
assumptions on lower bounds\textup{)}?

If that result is not true in general, is it true
if $M$ also belongs to the given variety~$\varV$?
\end{question}

A counterexample to either version of the above question would
probably have to be fairly complicated,
in view of Proposition~\ref{P.semilat_bdd_below}.

In a different direction, note that in our main result,
Theorem~\ref{T.main},
the assumption that $M$ had a lattice structure did not come into
the statement, except to make the concept of isotone map meaningful,
for which a structure of order would have sufficed; though the
lattice structure was used in the proof.
The same observation applies to many of our other results.
This suggests a family of questions.

\begin{question}\label{Q.M_not_lat}
For each of Lemmas~\ref{L.L&e} and~\ref{L.M'toM}
and Theorems~\ref{T.main},~\ref{T.complete}, and~\ref{T.retract},
does the same conclusion hold for a significantly
wider class of orders $M$ than the underlying orders of lattices?

Likewise, for Proposition~\ref{P.semilat_bdd_below}
and Theorem~\ref{T.semilat},
does the same conclusion hold for a significantly
wider class of join-semilattices $M$ than the underlying
join-semilattices of lattices?
\end{question}

When we showed in Section~\ref{S.alt_pfs}
that our main result could not be proved by a
construction with ``too much symmetry'', we called on the
fact that in a free lattice in the variety $\varL$ of all lattices,
no element is doubly reducible (both a proper meet and a proper join;
see sentence following display~(\ref{d.sym})).
Lattices (not necessarily free) with the latter property were
considered in~\cite{tfbl}.
We do not know the answer to

\begin{question}\label{Q.m-j-red}
Are there any nontrivial proper subvarieties $\varV$ of $\varL$ such
that in every free lattice $\Free_{\varV}(X)$, no element
is doubly reducible?
\end{question}

A final tantalizing question is,

\begin{question}\label{C.section}
In the situation of Corollary~\ref{C.VtoL},
can the isotone map $\Free_{\varV}\E L\to\Free\E L$
be taken to be a section \textup{(}left inverse\textup{)}
to the natural lattice homomorphism
$\Free\E L\to\Free_{\varV}\E L$?

In particular, for every variety of lattices
$\varV$ and every set $X$, does the natural lattice homomorphism
$\Free(X)\to\Free_{\varV}(X)$ admit an isotone section?
\end{question}


\end{document}